\documentclass[twoside,reqno,11pt]{amsart}

\usepackage[cp1251]{inputenc}
\usepackage{latexsym,amscd,amssymb,amsopn,amsthm,amsfonts,amsmath}
\usepackage{color}
\usepackage[dvipsnames]{xcolor} 

\theoremstyle{plain}
\newtheorem{theorem}{Theorem}[section]

\newtheorem{lemma}{Lemma}[section]
\newtheorem{corollary}{Corollary}[section]

\theoremstyle{definition}
\newtheorem{remark}{Remark}[section]

\begin{document}

\title[On the de Rham complex over weighted H\"older spaces]{On the cohomologies of the de Rham complex over weighted isotropic and anisotropic H\"older spaces}

\author{K.V. Gagelgans}

\address[Ksenia Gagelgans]
       {Siberian Federal University
\\
         Institute of Mathematics and Computer Science
\\
         pr. Svobodnyi 79
\\
         660041 Krasnoyarsk
\\
         Russia}
				
\email{ksenija.sidorova2017@yandex.ru}

\begin{abstract}
We consider the de Rham complex over scales of 
weighted isotropic and anisotropic H\"older spaces with prescribed asymptotic 
behaviour at the infinity.  
Starting from  theorems on the solvability of the system of operator equations generated by the de Rham differential $d$ and the operator $d^*$ formally adjoint to it, a description of the cohomology groups of the de Rham complex over these scales was obtained. It was also proved that in the isotropic case the cohomology space is finite-dimensional, and in the anisotropic case the general form of an element from the cohomology space is presented.
\end{abstract}



\maketitle 




\section*{Introduction}

It is well known that de Rham  cohomologies provide a lot of information 
on important invariants of a differentiable manifold, see for instance, \cite{deRh55}. 
However advances in the theory of differential equations and, 
more generally, in theory of complexes of differential operators, 
lead us to a slightly different concept: differential complexes 
over function spaces on the corresponding manifold 
(Hilbert spaces, Banach spaces, Fr\'echet spaces, etc.), 
see, for instance, \cite{Hoer65}, \cite{AT6}, \cite{Tark95a}. 
As applied to the de Rham complex, this concept can be traced in the theory of 
cohomologies over Lebesgue and Sobolev spaces, see, for instance, 
\cite{GKShw82}, \cite{Vod2010}. On this way, the classical de Rham 
cohomologies (over the Fr\'echet spaces of smooth functions) and 
cohomologies over Banach spaces may essentially differ.

This article is a logical continuation of papers \cite{SidShl19} and \cite{SidShl20}, 
in which we discussed solvability conditions for the 
operator equations generated by the differentials of the de Rham complex over isotropic 
weighted H\"older spaces on the Euclidean space ${\mathbb R}^n$ and the 
induced de Rham  differentials over anisotropic weighted H\"older spaces on 
the strip ${\mathbb R}^n \times [0,T]$ 
from the Euclidean space ${\mathbb R}^{n+1}$, where the variable $t\in [0,T]$, 
is considered as a parameter included in the coefficients of differential forms.
One of the motivations for studying the cohomology groups of the induced complex is the connection 
between the de Rham complex and various models of hydrodynamics, see, for example, 
\cite{BerMaj02}, \cite{ShlTa18}.

More precisely, using the method of integral representations, the technique of working in 
weighted spaces elaborated in articles \cite{NireWalk73}, \cite{McOw79},
and the results obtained in \cite{SidShl19} and \cite{SidShl20}, 
it is proved in the paper that the cohomology groups
of the de Rham complex over isotropic weighted 
H\"older spaces is finite-dimensional, and their dimension is indicated in dependence on the weight index of the space and the degree of the forms, and one of the possible 
bases in the cohomology space
 is indicated. It is pertinent to note that there is no countable basis in the induced cohomology 
space over non-separable weighted H\"older spaces.

\section{Isotropic case}
As usual, we denote by $\mathbb{R}^n$ the Euclidean space of dimension 
$n\geqslant 1$. Let us first consider the classical de Rham complex over the domain ${\mathcal X}$ in $\mathbb{R}^n$:
\begin{equation*}
0 \rightarrow  \Omega ^{0}({\mathcal X}) \xrightarrow{d^0}
\Omega ^{1}({\mathcal X}) \xrightarrow{d^1} 
\Omega ^{2}({\mathcal X}) \xrightarrow{d^2} \ldots  
\xrightarrow{d^{n-1}} \Omega ^{n}({\mathcal X}) \rightarrow 0;   
\end{equation*}
here 
$d^q$ are de Rham differentials given by exterior derivatives,
$\Omega ^{0}({\mathcal X})$ is the space of infinitely differentiable functions on 
${\mathcal X}$, 
and $\Omega ^{k}({\mathcal X})$ is the space of exterior differential
$k$-forms whose coefficients are smooth functions on ${\mathcal X}$.

As is well known, the cohomology groups of this complex are trivial 
for $0<q \leqslant n$ if ${\mathcal X}$ is a convex 
(or, more generally, star-shape) domain, see \cite{deRh55}.  
However, when considering the de Rham complex over other function spaces, the cohomology 
spaces may differ significantly from the classical ones, see, for example, 
\cite{GKShw82}, \cite{Vod2010}. 

In the papers \cite{SidShl19} and \cite{SidShl20} 
the following systems of 
operator  equations were considered in weighted H\"older spaces over $\mathbb{R}^n$:
\begin{equation*}
\begin{cases}
du=f, \\
d^*u=g,
\end{cases}
\end{equation*}
where $d$, as above, denotes
the de Rham differential, and $d^*$ is formally adjoint to it, and the weights regulate the 
decreasing order of the considered differential forms at the point at infinity. It turned out 
that as the order of decrease of the desired differential form increases, sufficiently 
restrictive conditions for the solvability of this system of operator equations appear. 
This led us to the question of describing the cohomology groups 
of the de Rham complex on the chosen spaces.

Let's describe the situation in more detail. Namely, we put
\begin{equation*}
   w (x)
 =
 \sqrt{1 + |x|^2}, \,\,
   w (x,y)
 =
 \max \big\{ w (x), w (y) \big\} , \,\, x, y \in \mathbb{R}^n.
\end{equation*}

For $s\in {\mathbb Z}_+$ and $\delta \in \mathbb{R}$
we denote by $C^{s,0}_\delta$ the space of $s$ 
times continuously differentiable functions in $\mathbb{R}^n$ with finite norm
\begin{equation*}
   \| u \|_{C^{s,0}_\delta}
 = \sum_{|\alpha| \leqslant s}
   \sup_{x \in \mathbb{R}^n}
   w^{\delta+|\alpha|} (x)
   \big|\partial^\alpha u (x)\big|.
\end{equation*}

Let $\mathcal{X}$ be some subset in $\mathbb{R}^n$. Then let $U \subset \mathbb{R}^n$ be a non-empty bounded neighborhood of zero in the topology $\mathbb{R}^n$, if $\mathcal{X}$ is not bounded, or an empty set otherwise.
For $0 < \lambda \leqslant 1$  we set
\begin{equation*}
   \langle u \rangle_{\lambda,\delta,\overline {\mathcal X}}
 = \sup_{x,y \in \overline {\mathcal X} \setminus U  ,  x \neq y \atop   |x-y| \leqslant |x|/2    }
   w ^{\delta+\lambda} (x,y) \frac{\big|u (x) - u (y)\big|}{|x-y|^\lambda}.
\end{equation*}

Let $C^{0,\lambda}_\delta$ consists of all continuous functions on
$\mathbb{R}^n$ with a finite norm
\begin{equation*}
   \| u \|_{C^{0,\lambda}_\delta}
 = \| u \|_{C^{0,\lambda} (\overline U)}
 + \| u \|_{C^{0,0}_\delta} + \langle u \rangle_{\lambda,\delta,\mathbb{R}^n}, 
\end{equation*}
where $\| \cdot \|_{C^{0,\lambda} (\overline U)} = \|\cdot \|_{C^{0,0} (\overline U)}
+ \langle \cdot \rangle_{\lambda,\overline U}$ is the norm of the usual 
H\"older space $C^{0,\lambda} (\overline U)$ on the compact $\overline U$. 
Finally, for $s \in \mathbb{Z}_+$, let $C^{s,\lambda}_\delta$ 
denotes the space of all $s$  times continuously differentiable functions on $\mathbb{R}^n$ with a finite norm 
\begin{equation*}
   \| u \|_{C^{s,\lambda}_{\delta}}
 = \sum_{|\alpha| \leqslant s}
   \| \partial^\alpha u \|_{C^{0,\lambda}_{\delta+|\alpha|} }.
\end{equation*}

We proceed to the construction of the de Rham complex over the selected Banach spaces. To this end, for the differential operator $A$, acting on differential forms over 
$\mathbb{R}^n$, we denote by
$C^{s,\lambda}_{\delta,\varLambda^{q}} \cap \mathcal{S}_A$
the space of differential forms 
$u \in C^{s,\lambda}_{\delta,\varLambda^{q}} $, satisfying
$Au = 0$ in the sense of distributions in $\mathbb{R}^n$. Obviously, this is a closed subspace in
$C^{s,\lambda}_{\delta,\varLambda^{q}}$, 
which means it is a Banach space with the induced norm.

Below we will consider the following complex of Banach spaces 
\begin{equation}
\label{complex1}
0 \rightarrow C^{s,\lambda}_{\delta, \varLambda^{0}} \xrightarrow{d_0}
C^{s-1,\lambda}_{\delta+1, \varLambda^{1}} \xrightarrow{d_1}\ldots \xrightarrow{d_{s-2}} 
C^{1,\lambda}_{\delta+(s-1), \varLambda^{s-1}} \xrightarrow{d_{s-1}}
C^{0,\lambda}_{\delta+s, \varLambda^{s}}  \cap \mathcal{S}_{d^{s}}  \xrightarrow{d_{s}} 0,
\end{equation}
where $d^q$ are de Rham differentials given by external derivatives, and
$\varLambda^{q}$ are bundles of exterior differential forms of degree $0 \leqslant q \leqslant n$ over $\mathbb{R}^n$. 

The first goal of this paper is to describe the cohomology groups 
of the complex (\ref{complex1}), therefore, let the set $Z^{s,\lambda}_{\delta, \varLambda^{q}}$ denotes the space of cocycles and consists of the forms contained in intersection
$C^{s,\lambda}_{\delta, \varLambda^{q}} \cap \mathcal{S}_{d^{q}}$, the set
$B^{s,\lambda}_{\delta, \varLambda^{q}}$ denotes the space of coboundaries and consists of forms 
$f \in C^{s,\lambda}_{\delta, \varLambda^{q}}$, for which exists
$u \in C^{s+1,\lambda}_{\delta-1, \varLambda^{q-1}}$, satisfying the equation $f=du$,
and the set 
$H^{s,\lambda}_{\delta, \varLambda^{q}}$ denotes the cohomology space, that is, is equal to set $Z^{s,\lambda}_{\delta, \varLambda^{q}} / B^{s,\lambda}_{\delta, \varLambda^{q}}$.

In order to pass to the description of cohomology spaces, it is necessary to 
cite some previous results, namely: in the paper \cite{SidShl19} continuous linear operators
\begin{equation}
 \label{eq.DR}
\left(d^{q}, (d^{q-1})^*\right): C^{s+1,\lambda}_{\delta,\varLambda^{q}} 
\to C^{s,\lambda}_{\delta+1 ,\varLambda^{q+1}}  \cap \mathcal{S}_{d^{q+1}}
 \oplus C^{s,\lambda}_{\delta+1, \varLambda^{q-1}}  \cap \mathcal{S}_{(d^{q-2})^*}, 
\end{equation}
\begin{equation}
 \label{eq.DR.hat}
\left(d^{q}, (d^{q-1})^*\right):  \mathfrak{C}^{s+1,\lambda}_{\delta,\varLambda^{q}} 
\to  \mathfrak{C}^{s,\lambda}_{\delta+1 ,\varLambda^{q+1}}  \cap \mathcal{S}_{d^{q+1}}
 \oplus  \mathfrak{C}^{s,\lambda}_{\delta+1, \varLambda^{q-1}}  \cap \mathcal{S}_{(d^{q-2})^*},
\end{equation}
were considered, where
$\mathfrak{C}^{s,\lambda}_{\delta,\varLambda^{q}}$ is the closure of the space ${\mathcal D} ({\mathbb R}^n)$ of infinitely differentiable functions with compact supports in ${\mathbb R}^n$ in the space 
$C^{s,\lambda}_{\delta,\varLambda^{q}}$. A more detailed description of this space was given in \cite[Theorem 2]{SidShl19}. Taking as 
$\mathcal{H}_{\leqslant m, \varLambda^{q}}$ the space of all differential forms of degree $q$, whose coefficients are harmonic polynomials of degree $\leqslant m$,  the following theorem was formulated and proved.

\begin{theorem}
\label{t.weight.Hoelder.d}
Let  
   $n \geqslant 2$,
   $s\in {\mathbb Z}_+$,  $0 < \lambda < 1$.
If   $\delta > 0$ and
   $\delta+1-n \not\in \mathbb{Z}_+$, then
the operators  \eqref{eq.DR} and  \eqref{eq.DR.hat} are Fredholm. Moreover,
\begin{itemize}
\item \eqref{eq.DR} and \eqref{eq.DR.hat} are isomorphisms if $0 < \delta < n-1$;
\item \eqref{eq.DR} and \eqref{eq.DR.hat} are injections with a closed image if 
$n-1+m < \delta < n+m$ for $m \in \mathbb{Z}_+$; 
more precisely, the image of the operator \eqref{eq.DR}
consists of all pairs
   $f \in  C^{s,\lambda}_{\delta+1,\varLambda^{q+1}} \cap 
\mathcal{S}_{d^{q+1}}$, 
	$g \in  C^{s,\lambda}_{\delta+1,\varLambda^{q-1}} \cap 
\mathcal{S}_{(d^{q-2})^*}$, 
satisfying
\begin{equation}
\label{eq.image}
(f ,d^q h)_{L^2 _{\varLambda^{q+1}}({\mathbb R}^n)} +
(g ,(d^{q-1})^* h)_{L^2_{\varLambda^{q-1}} ({\mathbb R}^n)}= 0 
\mbox{ for all } h \in \mathcal{H}_{\leqslant m+1, \varLambda^{q}},
\end{equation}
and the image of the operator \eqref{eq.DR.hat}
consists of all pairs
   $f \in  \mathfrak{C}^{s,\lambda}_{\delta+1,\varLambda^{q+1}} \cap 
\mathcal{S}_{d^{q+1}}$, 
	$g \in  \mathfrak{C}^{s,\lambda}_{\delta+1,\varLambda^{q-1}} \cap 
\mathcal{S}_{(d^{q-2})^*}$, 
satisfying \eqref{eq.image}.
\end{itemize}
\end{theorem}

The main tools in the proof of the theorem were the following defined integral operators.
Namely, let
\begin{equation*}
(E_q u) (x) = \int\limits_{\mathbb{R}^n} u(y) \wedge e_q (x,y)\, \, 
\end{equation*}
for a suitable $q$-form $u$, where 
$$
e_q (x,y) = \sum_{|I|=q} e (x-y) \, (\star dy_I) \, dx_I, \,\, 
   e (x)
 = \left\{ \begin{array}{lcl}
             \displaystyle
             \frac{1}{\pi} \ln |x|,
           & \mbox{for}
           & n = 2,
\\
             \displaystyle
             \frac{1}{\sigma_n} \frac{|x|^{2-n}}{2-n},
           & \mbox{for}
           & n \geqslant 3,
           \end{array}
   \right.
$$
that is, $e$ is a standard fundamental solution of convolution type for the Laplace operator in $\mathbb{R}^n$, and $e_q$ is its analogue for action on exterior differential forms
(here $\sigma_n$ is the area of the unit sphere in $\mathbb{R}^n$). 
Now, for $f \in C^{s,\lambda}_{ \delta+1,\varLambda^{q+1}}$, 
$g \in C^{s,\lambda}_{\delta+1, \varLambda^{q-1}} $, we define
\begin{equation} 
\label{eq.volume.potential.d} 
(\varPhi _q \, f) (x) =
\int\limits_{{\mathbb R}^n} f (y)  \wedge \phi_q (x,y), 
\, \, (\hat \varPhi _q \, g) (x) =
\int\limits_{{\mathbb R}^n} g (y)  \wedge \hat \phi_q (x,y),
\end{equation}
where
\begin{equation*}
\phi_q (x,y)= (d^{n-q-1})^*_y e_q (x,y) , \,\, 
\hat \phi_q (x,y) =d^{n-q}_y e_q (x,y) , \quad n\geqslant 2.
\end{equation*}

Next, we use the following decomposition of the fundamental solution of the Laplace operator in homogeneous harmonic polynomials 
$\left\{ h_k^{(j)}\right\}$, forming an orthonormal basis in $L^2 (\partial B_1)$ on the unit sphere
$\partial B_1$ in ${\mathbb R}^n$:
\begin{equation*}
   e (x-y)
 = e (x-0)
 - \sum_{k=1}^\infty
   \sum_{j=1}^{J (k)}
   \frac{h_k^{(j)} (x) h_k^{(j)} (y)}{(n+2k-2) |x|^{n+2k-2}},
\end{equation*}
here $n \geqslant 2$, $k$ is the degree of homogeneity of the polynomial, $j$ is the number of a homogeneous polynomial of degree $k$ in the basis, and the series converges uniformly together with all derivatives on compact sets of their cone
 $\left\{ |x| > |y| \right\}$ in $\mathbb{R}^{2n}$
   (see for example  \cite{McOw79}, \cite{Shla92}). We put	
\begin{equation*}
\phi_{m,q} (x,y) = \phi_q (x,y) +  \sum_{\left|I\right|=q}\sum_{k=1}^{m+1} \sum_{j=1}^{J (k)}
\frac{h_k^{(j)} \left(x\right) \left(d^{n-q-1}\right)^*_y \left(h_k^{(j)} (y) (\star dy_I)\right)  dx_I}{(n+2k-2){
\vartheta}^{n+2k-2}(x)}, 
\end{equation*}
\begin{equation*} 
\hat \phi_{m,q} (x,y) = \hat \phi_q (x,y) +  \sum_{\left|I\right|=q}\sum_{k=1}^{m+1} \sum_{j=1}^{J (k)}
\frac{h_k^{(j)} \left(x\right) \left(d^{n-q} \right)_y 
\left(h_k^{(j)} (y) (\star dy_I)\right) dx_I  }{(n+2k-2){\vartheta}^{n+2k-2}(x)}, 
\end{equation*}
where 
$m\in {\mathbb Z}_+$ and
${\vartheta} (x)$ is some smooth function such that
${\vartheta} (x) = |x|$ for $|x| \geqslant 2$ and $0< 1/{\vartheta}(x) \leqslant 1$. 

Then we define the potential $\varPhi_{m,q}$ as follows
\begin{equation*}
 \varPhi_{m,q} f (x)=\int\limits_{{\mathbb R}^n} f (y) \wedge  \phi_{m,q} (x,y)  ,
\,\, 
\hat  \varPhi_{m,q} g (x)=\int\limits_{{\mathbb R}^n} g (y) \wedge \hat \phi_{m,q} (x,y) . 
\end{equation*}

The main property of 
these potentials is the following.
\begin{lemma}
\label{l.phi.Hoelder.d}
If $\delta > 0$, $\lambda \in (0,1)$, then the potentials $\varPhi_q f$, $\hat \varPhi_q g$,  
given by formula (\ref{eq.volume.potential.d}), satisfy, in the sense of distributions on $\mathbb{R}^n$,  equalities 
\begin{align*} 
   d^q (\varPhi_q f) = f, && (d^{q-1})^* (\varPhi_q f) = 0, &&
	(d^{q-1})^* (\hat \varPhi_q g) = g, && (d^{q}) (\hat \varPhi_q g) = 0 
	\end{align*}
for all 
   $f \in C^{0,\lambda}_{\delta+1,\varLambda^{q+1}} 
	\cap \mathcal{S}_{d^{q+1}}$, 
	$g \in C^{0,\lambda}_{\delta+1, \varLambda^{q-1}}
	\cap \mathcal{S}_{(d^{q-2})^*}$.
Moreover, if $0 < \delta < n-1$, then the potentials (\ref{eq.volume.potential.d}) induce the bounded linear operators 
\begin{align*}
   \varPhi_q : C^{0,\lambda}_{\delta+1,\varLambda^{q+1}} 
	\to C^{1,\lambda}_{\delta, \varLambda^{q}} 
	, 
&&
   \hat \varPhi_q : C^{0,\lambda}_{\delta+1, \varLambda^{q-1}} 
	\to C^{1,\lambda}_{\delta, \varLambda^{q}} 
.
\end{align*}

\end{lemma}
Let us extract from Theorem \ref{t.weight.Hoelder.d} and 
Lemma \ref{l.phi.Hoelder.d} a rather expected consequence; we will use to describe the cohomology spaces of  complex \eqref{complex1}.
\begin{corollary}
\label{t.weight.Hoelder.d.cor}
Let $n \geqslant 2$, $s\in {\mathbb Z}_+$,  $0 < \lambda < 1$, 
$\delta > n/2$ and 
$f \in Z^{s,\lambda}_{\delta+1, \varLambda^{q+1}}$. If the form $u \in C^{s+1,\lambda}_{\delta, \varLambda^{q}}$
is a solution to the equation $du=f$,
then there is a form $v \in C^{s+1,\lambda}_{\delta, \varLambda^{q}}$,
satisfying 
\begin{equation}
\label{uchch}
\begin{cases}
dv=f, \\
d^*v=0.
\end{cases}
\end{equation}
\end{corollary}

\begin{proof}
Let the form $u$
satisfy the conditions of the corollary. Then (see \cite[Corollary 1]{SidShl19}) the following decomposition takes place:
\begin{equation}
\label{uch}
u = d^* \hat \varPhi u + d \varPhi u.
\end{equation}

Applying the operator $d$ to identity (\ref{uch}) and taking into account that $d\circ d =0$, we see that $v=d^*\hat \varPhi u$ is also a solution to the equation $du=f$.
And since $d^* \circ d^* =0$,
then $v$ is a solution to the system (\ref{uchch}).
Since the operator $d^*\hat \varPhi $ is continuous in $C^{s+1,\lambda}_{\delta, \varLambda^{q}}$ (also, by Corollary 1), it follows that 
$v \in C^{s+1,\lambda}_{\delta, \varLambda^{q}}$.
\end{proof}

Finally, we can go directly to the description of the cohomology groups of complex 
(\ref{complex1}). It follows from Theorem \ref{t.weight.Hoelder.d} that 
they are trivial for 
$0 < \delta < n-1$. Therefore, below we will consider the case when $n+m-1 < \delta < n+m$.

\begin{theorem}
\label{kohom1}
Let $n \geqslant 2$, $s\in {\mathbb Z}_+$,  $0 < \lambda < 1$ and $n+m-1 < \delta < n+m$. Then the cohomology groups of complex \eqref{complex1} are finite-dimensional and isomorphic to the image of the operator $d\left(\varPhi-\varPhi_m\right)$, acting from $Z^{s,\lambda}_{\delta+1, \varLambda^{q+1}}$
to $Z^{s,\lambda}_{\delta+1, \varLambda^{q+1}}$.
\end{theorem}

\begin{proof}
First, we prove that $d\left(\varPhi-\varPhi_m\right)$ is a continuous operator. 
Since $d$
acts continuously from $C^{s+1,\lambda}_{\delta, \varLambda^{q}}$
to $C^{s,\lambda}_{\delta+1, \varLambda^{q+1}}$,
and $\varPhi_m$
acts continuously from $C^{s,\lambda}_{\delta+1, \varLambda^{q+1}}$
to $C^{s+1,\lambda}_{\delta, \varLambda^{q}}$ (see \cite{SidShl19}),
we have
\begin{equation*}
\left\| d\left( \varPhi_m f \right) \right\|_
{C^{s,\lambda}_{\delta+1, \varLambda^{q+1}}} \leqslant
C_1 \left\| \varPhi_m f \right\|_
{C^{s+1,\lambda}_{\delta, \varLambda^{q}}} \leqslant
C_1 C_2\left\| f \right\|_
{C^{s,\lambda}_{\delta+1, \varLambda^{q+1}}}.
\end{equation*}

It follows from Theorem \ref{t.weight.Hoelder.d}
that the operator $\varPhi$
for $n+m-1 < \delta < n+m$
does not act continuously from $C^{s,\lambda}_{\delta+1, \varLambda^{q+1}}$ to $C^{s+1,\lambda}_{\delta, \varLambda^{q}}$.
But since the operator $\varPhi$ under the conditions of the theorem being proved is considered on $Z^{s,\lambda}_{\delta+1, \varLambda^{q+1}}$, then $d\varPhi f=f$. Thus, we get
\begin{multline*}
\left\| d\left(\varPhi f\right) - d\left(\varPhi_m f\right) \right\|_
{C^{s,\lambda}_{\delta+1, \varLambda^{q+1}}} \leqslant
\left\| d\left(\varPhi f\right) \right\|_
{C^{s,\lambda}_{\delta+1, \varLambda^{q+1}}} + \\
+\left\| d\left(\varPhi_m f\right) \right\|_
{C^{s,\lambda}_{\delta+1, \varLambda^{q+1}}} 
\leqslant \left( 1+C_1 C_2 \right) \left\|f\right\|_
{C^{s,\lambda}_{\delta+1, \varLambda^{q+1}}}. 
\end{multline*}

The next step in the proof is to establish the equality of spaces
$C^{s,\lambda}_{\delta+1, \varLambda^{q+1}} \cap \mathcal {S}_{d\left(\varPhi-\varPhi_m\right)}$  and
$B^{s,\lambda}_{\delta+1, \varLambda^{q+1}}$.
Suppose
\begin{equation}
\label{imimim}
d\left(\varPhi-\varPhi_m\right)f = 0.
\end{equation}

As we noted above, $d\varPhi f=f$. Then it follows from equation (\ref{imimim}) that $f$
can be represented as $f=d\varPhi_m f$.
Since the operator $\varPhi_m$
acts continuously (see \cite[Lemma 5]{SidShl19}) from $C^{s,\lambda}_{\delta+1, \varLambda^{q+1}}$
to $C^{s+1,\lambda}_{\delta, \varLambda^{q}}$, then $f \in B^{s,\lambda}_{\delta+1, \varLambda^{q+1}}$.

Suppose now that $f \in B^{s,\lambda}_{\delta+1, \varLambda^{q+1}}$. Then there is a form $u \in C^{s+1,\lambda}_{\delta, \varLambda^{q}}$
satisfying the condition $f=du$.
Corollary \ref{t.weight.Hoelder.d.cor} implies that in this case there also exists $v \in C^{s+1,\lambda}_{\delta, \varLambda^{q}}$,
satisfying system (\ref{uchch}).
Then it follows from Theorem \ref{t.weight.Hoelder.d} that the form $f$
is orthogonal to any element of the set ${\mathcal H}^{q}_{\leqslant m+1}$,
and hence $d\left(\varPhi-\varPhi_m\right)f = 0$.

The last part of the proof consists in constructing the desired isomorphism. Consider a map $\omega$
that assigns to each class $\left[g\right]$ in $H^{s,\lambda}_{\delta+1, \varLambda^{q+1}}$ an element in the image of an operator $d(\varPhi-\varPhi_m)$, 
acting from $Z^{s,\lambda}_{\delta+1, \varLambda^{q+1}}$ to $Z^{s,\lambda}_{\delta+1, \varLambda^{q+1}}$. Let this map acts according to rule
\begin{equation*}
\omega\left( \left[g\right] \right) = d(\varPhi-\varPhi_m)g,
\end{equation*}
where $g$ is some representative of the class $\left[g\right]$.

Let us show that $\omega$
does not depend on the choice of a representative of the class $\left[g\right]$. Let $g_1, g_2 \in \left[g\right]$ and $g_1\neq g_2$, 
then $g_2$ can be represented in the form $g_2=g_1+du$, where $du$ is an element of the space $B^{s,\lambda}_{\delta+1, \varLambda^{q+1}}$. We obtain
\begin{equation*}
d(\varPhi-\varPhi_m)g_2=
d(\varPhi-\varPhi_m)(g_1+du)=
d(\varPhi-\varPhi_m)g_1+
d(\varPhi-\varPhi_m)(du).
\end{equation*}

But $du$ is contained in the set $B^{s,\lambda}_{\delta+1, \varLambda^{q+1}}$, which means that it is also contained in $C^{s,\lambda}_{\delta+1, \varLambda^{q+1}} \cap \mathcal {S}_{d\left(\varPhi-\varPhi_m\right)}$ . Thus 
\begin{equation*}
d(\varPhi-\varPhi_m)g_2=
d(\varPhi-\varPhi_m)g_1.
\end{equation*}

Next, we turn to the proof of the injectivity of the map $\omega$. 
Let $g_1 \in \left[g_1\right]$, $g_2 \in \left[g_2\right]$ and $\left[g_1\right] \neq \left[g_2\right]$. 
Suppose $\omega\left(\left[g_1\right]\right) = \omega\left(\left[g_2\right]\right)$. We obtain
\begin{eqnarray*}
&\omega\left(\left[g_1\right]\right) = d(\varPhi-\varPhi_m)g_1 = g_1 + d\varPhi_m g_1,& \\
&\omega\left(\left[g_2\right]\right) = d(\varPhi-\varPhi_m)g_2 = g_2 + d\varPhi_m g_2.&
\end{eqnarray*}

Then, if $\omega\left(\left[g_1\right]\right) = \omega\left(\left[g_2\right]\right)$ then $g_1 + d\varPhi_m g_1 = g_2 + d\varPhi_m g_2$ , and therefore
$g_1$ and $g_2$
belong to the same class, which contradicts the condition.

Finally, it remains to prove that the mapping is surjectivity. If a form $f$ is an element of the image of an operator $d(\varPhi-\varPhi_m)$,
then there exists a form $g \in Z^{s,\lambda}_{\delta+1,\varLambda^{q+1}}$ such that $f=d(\varPhi-\varPhi_m)g$ and hence there is also a class $\left[g\right]$ containing this form. 
Thus, the mapping $\omega$ is the isomorphism we were looking for.
\end{proof}

\begin{remark} Note that the operator $\left(\varPhi - \varPhi_{m}\right)$ has the form
\begin{equation}
\label{dlaremar}
\left(\varPhi - \varPhi_{m}\right) f (x)=\int\limits_{{\mathbb R}^n} f (y) \wedge  
\sum_{\left|I\right|=q} \sum_{k=1}^{m+1} \sum_{j=1}^{J (k)}
\frac{h_k^{(j)} (x) (d^{n-q-1})^*_y (h_k^{(j)} (y) (\star dy_I))  dx_I}{(n+2k-2){
\vartheta}^{n+2k-2}(x)}.
\end{equation}
Thus, all elements of the image of the operator 
$\left(\varPhi - \varPhi_{m}\right)$ 
are represented as finite linear combinations of basis vectors of the space 
${\mathcal H}_{\leqslant m, \Lambda^q}$ of $q$-forms with coefficients in 
the space ${\mathcal H}_{\leqslant m}$, divided by functions
$\vartheta^{n+2k-2}(x)$.
Applying operator $d$ to the right-hand side of (\ref{dlaremar}), we see that the dimension of the image of operator $d(\varPhi - \varPhi_{m})$ does not exceed the dimension of the image of $\left(\varPhi - \varPhi_{m}\right)$. 
\end{remark}

\section{Anisotropic case}

Similarly to 
the usual anisotropic H\"older spaces (see, for example, \cite{Kry96} or \cite{LadSoUr67}), 
anisotropic weighted H\"older spaces on a cylindrical domain ${\mathcal X}\times (0,T)$ were 
introduced in paper \cite{ShlTa18} (see also \cite{Kond66} as well as 
\cite{Be11} for the case when the base 
$\mathcal X$ is a Riemannian manifold with a conical singularity or \cite{MazRoss04} for polyhedral domains). 
The specificity of these spaces is that, although the 
dilation principle with respect to 
the element smoothness index is fulfilled, it is violated with respect to the weight index. 
This is done intentionally in order to guarantee, under certain conditions, the continuity of 
both parabolic and elliptic potentials on the scale of these spaces, see 
\cite[\S 3,b \S 4]{ShlTa18}.  

Consider a cylinder $\mathbb{R}^n \times [0,T]$, where $T>0$. On this set, we construct a space $C^{0,0,0,0}_{\delta,T}$ of continuous functions with a weight only in the variables $(x_1,…,x_n)$ and with a finite norm
\begin{equation*}
\|u \|_{C^{0,0,0,0} _{\delta,T}} =  \sup_{t \in [0,T]} 
\|u (\cdot, t)\|_{C^{0,0} _\delta}.
\end{equation*}

For $0<\lambda \leqslant 1$ and $0<\mu \leqslant 1$  we define weighted H\"older spaces  $C^{0,0,\lambda,\mu} _{\delta,T}$, for which the norm is finite
\begin{equation*}
\|u \|_{C^{0,0,\lambda,\mu} _{\delta,T}} =  \sup_{t \in [0,T]} 
\|u (\cdot, t)\|_{C^{0,0} _\delta } + \langle u\rangle 
_{\lambda, \mu,\delta,\mathbb{R}^n ,T},
\end{equation*}
where
\begin{equation*}
\langle u\rangle_{\lambda,\mu, \delta, \mathbb{R}^n,T} =
\left\{ 
\begin{array}{ll}
\!\!\!\!\sup\limits_{t \in [0,T]} \langle u(\cdot,t)\rangle_{\lambda,\delta, \mathbb{R}^n},
&\!\!\! \mu=0,
\\
\!\!\!\!\sup\limits_{t \in [0,T]} \langle u(\cdot,t)\rangle_{\lambda,\delta, \mathbb{R}^n}
+ \!\!\!\! \sup\limits_{t\ne \tau \atop 
t,\tau \in [0,T]}\frac{  \|u (\cdot,t) - u (\cdot,\tau)\|_{C^{0,0}_\delta }}
{|t-\tau|^\mu}, & \!\!\!\mu\in (0,1].
\end{array}
\right.
\end{equation*}

Further, let $s \in \mathbb{Z}_+$. We denote by $C^{2s,s,\lambda, \mu}_{\delta,T}$ the spaces of $s$ times continuously differentiable functions over ${\mathbb R}^n \times [0,T]$, having a finite norm
\begin{equation*}
\|u \|_{C^{2s,s,\lambda,\mu} _{\delta,T} } = \sum_{
 |\alpha| +2j \leq 2s}  \| \partial^j_t \partial^{\alpha}_x u 
\|_{C^{0,0,\lambda,\mu}_{\delta + |\alpha|,T} }. 
\end{equation*}

Finally, for $k \in \mathbb{Z}$ we define spaces 
$C^{2s+k,s,\lambda,\mu} _{\delta,T}$ with additional smoothness in the variables $(x_1,…,x_n)$ with a finite norm
\begin{equation*}
\|u \|_{C^{2s+k,s,\lambda,\mu} _{\delta,T}} =
\sum_{|\beta|\leq k } 
 \|\partial^{\beta}_x u \|_{C ^{2s,s,\lambda,\mu}_
{\delta+|\beta|,T}}.
\end{equation*}

\subsection{Case $\mu=0$}

Now let $\varLambda ^q (t)$ denote the induced bundle of exterior differential forms of 
degree $q$ over the half-space ${\mathbb R}^{n+1}_{t\geq 0} ={\mathbb R}^{n} \times 
[0,+\infty)$ with coordinates $(x,t)$,  that is, its sections are 
exterior differential forms on ${\mathbb R}^{n} $, whose coefficients depend 
on the parameter $t \in [0,+\infty)$:
\begin{equation*}
U= \sum_{|I|=q} U_I (x,t) dx_I, \, I= (i_1, \dots i_q), \, 1\leqslant i_j \leqslant n,  \, \, 
0\leqslant q \leqslant n ,
\end{equation*}
where, as usual  $dx_I = dx_{i_1} \wedge \dots \wedge dx_{i_q}$, and the symbol
$\wedge$ denotes the exterior product of differential forms.

Consider the following complex of Banach spaces:
\begin{multline}
\label{complex2}
0 \rightarrow C^{2s+k,s,\lambda,0}_{\delta, \varLambda^{0}} \xrightarrow{d_0}
C^{2s+k-1,s,\lambda,0}_{\delta+1, \varLambda^{1}} \xrightarrow{d_1}\ldots \\
\ldots 
\xrightarrow{d_{k-2}} 
C^{2s+1,s,\lambda,0}_{\delta+(k-1), \varLambda^{k-1}} \xrightarrow{d_{k-1}}
C^{2s,s,\lambda,0}_{\delta+k, \varLambda^{k}}  \cap \mathcal{S}_{d^{k}}  \xrightarrow{d_{k}} 0,
\end{multline}
 where
\begin{equation*}
(d^q U) (x,t) = \sum_{j=1}^n \sum_{|I|=q} \frac{\partial 
U_I (x,t)}{\partial x_j } dx_j \wedge dx_I  .
\end{equation*}

In the article \cite{SidShl20} Theorem \ref{t.weight.Hoelder.d} was extended to bounded operators  
\begin{equation} 
\label{eq.d.easy}
\left(d^q, (d^{q-1})^*\right): 
C^{2s+k+1,s,\lambda,0}_{\delta,T,\varLambda^{q}} 
\to C^{2s+k,s,\lambda,0}_{\delta+1,T,\varLambda^{q+1} }  \cap \mathcal{S}_{d}
 \oplus C^{2s+k,\lambda,0}_{\delta+1,T,\varLambda^{q-1}}  \cap 
\mathcal{S}_{d^*}, \quad k \in {\mathbb Z}_+, 
\end{equation}
\begin{equation} 
\label{eq.d.easy.closure}
\left(d^q, (d^{q-1})^*\right): 
\mathfrak{C}^{2s+k+1,s,\lambda,0}_{\delta,T,\varLambda^{q}} 
\to \mathfrak{C}^{2s+k,s,\lambda,0}_{\delta+1,T,\varLambda^{q+1} }  \cap \mathcal{S}_{d}
 \oplus \mathfrak{C}^{2s+k,\lambda,0}_{\delta+1,T,\varLambda^{q-1}}  \cap 
\mathcal{S}_{d^*}, \quad k \in {\mathbb Z}_+,
\end{equation}
where $\mathfrak{C}^{2s+k,s,\lambda,0}_{\delta,T}$ denotes the closure of the set 
$C^\infty ([0,T],{\mathcal D} ({\mathbb R}^n))$
in the space ${C}^{2s+k,s,\lambda,0}_{\delta,T}$, and
$C^\infty ([0,T],{\mathcal D} ({\mathbb R}^n))$, in turn, is the space of infinitely differentiable mappings  $U:[0,T] \to {\mathcal D} ({\mathbb R}^n)$.
Thus, we have

\begin{corollary} 
\label{c.weight.Hoelder.d.easy}
Let
   $n \geq 2$,
   $s\in {\mathbb Z}_+$,  $0 < \lambda < 1$.
If   $\delta > 0$ and 
   $\delta+1-n \not\in \mathbb{Z}_+$, then the operators  \eqref{eq.d.easy} and \eqref{eq.d.easy.closure} are normal
solvable. Moreover,
\begin{itemize}
\item
\eqref{eq.d.easy} and  \eqref{eq.d.easy.closure} are isomorphisms if $0 < \delta < n-1$;

\item
\eqref{eq.d.easy} and  \eqref{eq.d.easy.closure}  are  injections with a closed image if
$n-1+m < \delta < n+m$ for $m \in \mathbb{Z}_+$; 
more precisely, the image of the operator \eqref{eq.d.easy}
consists of all pairs
   $f \in  C^{2s+k,s,\lambda,0}_{\delta+1,T,\varLambda^{q+1}} \cap 
\mathcal{S}_{d}$, 
	$g \in  C^{2s+k,s,\lambda,0}_{\delta+1,T,\varLambda^{q-1}} \cap 
\mathcal{S}_{d^*}$, 
satisfying

\begin{equation}
\label{eq.image.T}
(f (\cdot , t),d h)_{L^{2}_{\varLambda^{q+1}} ({\mathbb R}^n)} +
(g (\cdot, t) ,d^* h)_{L^{2} _{\varLambda^{q-1}}({\mathbb R}^n)}= 0 
\end{equation}
for all $t \in [0,T] \mbox{ and } h \in H^q_{\leqslant m+1}$,  
and the image of the operator \eqref{eq.d.easy.closure}
consists of all pairs
   $f \in  \mathfrak{C}^{2s+k,s,\lambda,0}_{\delta+1, T,\varLambda^{q+1}} \cap 
\mathcal{S}_{d}$, 
	$g \in  \mathfrak{C}^{2s+k,s,\lambda,0}_{\delta+1, T,\varLambda^{q-1}} \cap 
\mathcal{S}_{d^*}$, satisfying \eqref{eq.image.T}.
\end{itemize}
\end{corollary}

As in the case of isotropic spaces, we define cocycles, coboundaries, and cohomology as follows. Let $Z^{2s+k,s,\lambda,0}_{\delta, \varLambda^{q}}$ denotes the space of cocycles and contain forms belonging to the set
$C^{2s+k,s,\lambda,0}_{\delta, \varLambda^{q}} \cap \mathcal{S}_{d^{q}}$, and
$B^{2s+k,s,\lambda,0}_{\delta, \varLambda^{q}}$ denotes the space of coboundaries and contains elements
$f \in C^{2s+k,s,\lambda,0}_{\delta, \varLambda^{q}}$, for which there exists a form 
$u \in C^{2s+k+1,s,\lambda,0}_{\delta-1, \varLambda^{q-1}}$ such that the equation
$f=du$ is valid. Finally, let $H^{2s+k,s,\lambda,0}_{\delta, \varLambda^{q}}$ denotes the cohomology space of complex (\ref{complex2}) and contains forms lying in
$Z^{2s+k,s,\lambda,0}_{\delta, \varLambda^{q}} / B^{2s+k,s,\lambda,0}_{\delta, \varLambda^{q}}$.

We now get the following statement.
\begin{corollary}
\label{c.weight.Hoelder.d.easy.cor}
Let $n \geqslant 2$, $s,k\in {\mathbb Z}_+$,  $0 < \lambda < 1$, 
 $\delta > n/2$ and $f \in Z^{2s+k,s,\lambda,0}_{\delta+1, \varLambda^{q+1}}$. If the form
$u \in C^{2s+k+1,s,\lambda,0}_{\delta, \varLambda^{q}}$
is a solution to the equation $du=f$,
where 
then there exists a form $v \in C^{2s+k+1,s,\lambda,0}_{\delta, \varLambda^{q}}$,
satisfying the system of operator equations
\begin{equation*}
\begin{cases}
dv=f, \\
d^*v=0.
\end{cases}
\end{equation*}
\end{corollary}

\begin{proof}
Carrying out the proof as in Corollary \ref{t.weight.Hoelder.d.cor} and relying on Corollary 3.4 of \cite{SidShl20}, we obtain the assertion of the corollary.
\end{proof}
Finally, we proceed to describe the cohomology spaces 
of complex (\ref{complex2}).
\begin{theorem}
\label{kohom2}
Let $n \geqslant 2$, $s\in {\mathbb Z}_+$,  $0 < \lambda < 1$ and $n+m-1 < \delta < n+m$. Then the cohomology groups of complex \eqref{complex2} are isomorphic to the image of the operator $d\left(\varPhi-\varPhi_m\right)$, acting from $Z^{2s+k,s,\lambda,0}_{\delta+1, \varLambda^{q+1}}$
to $Z^{2s+k,s,\lambda,0}_{\delta+1, \varLambda^{q+1}}$.
\end{theorem}

\begin{proof}
The proof is carried out similarly to the proof of Theorem \ref{kohom1}, relying on Corollary \ref{c.weight.Hoelder.d.easy} instead of Theorem \ref{t.weight.Hoelder.d}.
\end{proof}

Note that the operator $\left(\varPhi - \varPhi_{m}\right)$ has the form
\small
\begin{equation*}
\left(\varPhi - \varPhi_{m}\right) f (x,t) = 
\int\limits_{{\mathbb R}^n} f (y,t) \wedge  
\sum_{\left|I\right|=q}\sum_{k=1}^{m+1} \sum_{j=1}^{J (k)}
\frac{h_k^{(j)} \left(x\right) \left(d^{n-q-1}\right)^*_y \left(h_k^{(j)} (y) (\star dy_I)\right)  dx_I}{(n+2k-2){
\vartheta}^{n+2k-2}(x)}.
\end{equation*}
\normalsize

Thus, the elements of the image of the operator $d\left( \varPhi-\varPhi_m \right)$
acting from $Z^{2s+k,s,\lambda,0}_{\delta+1, \varLambda^{q+1}}$ 
to $Z^{2s+k,s,\lambda,0}_{\delta+1, \varLambda^{q+1}}$ are represented as follows:
\begin{equation*}
d\left(
\sum_{\left|I\right|=q}\sum_{k=1}^{m+1} \sum_{j=1}^{J (k)} 
\frac{a_I (t) h_k^{(j)} \left(x\right) dx_I}
{(n+2k-2){\vartheta}^{n+2k-2}(x)} 
\right),
\end{equation*}
where $a_I(t)$ are functions of a variable $t$ of class $C^{s,0}([0,T])$.

\subsection{Case $\mu=\lambda/2$}

A direct generalization of Corollary \ref{c.weight.Hoelder.d.easy} to the scale of spaces
$C^{2s+k,s,\lambda,\frac{\lambda}{2}} _{\delta,T,\varLambda^q}$ is impossible, since the operators $\varPhi$,
$\hat \varPhi$ do not act continuously on it. Therefore, we need to 
change 
slightly the definitions of the spaces.
To this end, we denote by 
$\Gamma^{2s+k,s,\lambda,\frac{\lambda}{2}}_{\delta,T,\varLambda^q}$ the completion of the space 
$C^{2s+k,s,\lambda,\frac{\lambda}{2}}_{\delta,T,\varLambda^q}$ 
with respect to the norm
\begin{equation*}
\left\|u\right\|_{\Gamma^{2s+k,s,\lambda,\frac{\lambda}{2}}_{\delta,T,\varLambda^q}} = 
\left\|u\right\|_{C^{2s+k,s,\lambda,\frac{\lambda}{2}}_{\delta,T,\varLambda^q}}+
\left\|du\right\|_{C^{2s+k,s,\lambda,\frac{\lambda}{2}}_{\delta+1,T,\varLambda^{q+1}}}+
\left\|d^* u\right\|_{C^{2s+k,s,\lambda,\frac{\lambda}{2}}_{\delta+1,T,\varLambda^{q-1}}}.
\end{equation*}

Now we consider an operator of the following form
\begin{equation} 
\label{gtog}
(d, d^*): 
\Gamma^{2s+k,s,\lambda,\frac{\lambda}{2}}_{\delta,T,\varLambda^q} \rightarrow
\Gamma^{2s+k,s,\lambda,\frac{\lambda}{2}}_{\delta+1,T,\varLambda^{q+1}}  
\cap \mathcal{S}_{d} \oplus
\Gamma^{2s+k,s,\lambda,\frac{\lambda}{2}}_{\delta+1,T, \varLambda^{q-1}}  
\cap \mathcal{S}_{d^*}
\end{equation}
and formulate for it an analogue of Theorem \ref{t.weight.Hoelder.d} and Corollary \ref{c.weight.Hoelder.d.easy}.

\begin{theorem}
\label{c.weight.Hoelder.d.not.easy}
Let  
   $n \geq 2$,
   $s\in {\mathbb Z}_+$,  $0 < \lambda < 1$.
If    $\delta > 0$ and 
   $\delta+1-n \not\in \mathbb{Z}_+$, then the operator (\ref{gtog}) is normal
solvable. Moreover,
\begin{itemize}
\item
\eqref{gtog} is an isomorphism if  $0 < \delta < n-1$;
\item
\eqref{gtog} is an injection with a closed image if
$n-1+m < \delta < n+m$ for $m \in \mathbb{Z}_+$; 
more precisely, the image of the operator  \eqref{gtog} is the set of forms $(f,g) \in \Gamma^{2s+k,s,\lambda, \frac{\lambda}{2}}_{\delta+1,T,\varLambda^{q+1}\oplus \varLambda^{q-1}}$,
for which conditions
\begin{eqnarray}
\label{ortH}
&df(x,t)=0,&\\ \nonumber
&d^*g(x,t)=0,&
\end{eqnarray}
\begin{equation}
\label{ortH2}
\big(f(\cdot,t), dh\big)_{L_2}+\big(g(\cdot,t), d^*h\big)_{L_2}=0,
\end{equation}
are satisfied, 
and where \eqref{ortH2} is true for any $h$ from ${\mathcal H}_{\leqslant m+1}$.
\end{itemize}
\end{theorem}

\begin{proof}
It is easy to check that for $0<\delta<n-1$ the operator (\ref{gtog})
acts continuously from $\Gamma^{2s+k,s,\lambda,\frac{\lambda}{2}}_{\delta,T,\varLambda^q}$ to 
$\Gamma^{2s+k,s,\lambda,\frac{\lambda}{2}}_{\delta+1,T,\varLambda^{q+1} \oplus \varLambda^{q-1}}  \cap \mathcal{S}_{d}$. Let us prove that the operators $\varPhi$ and $\hat{\varPhi}$, acting from 
$\Gamma^{2s+k,s,\lambda,\frac{\lambda}{2}}_{\delta+1,T,\varLambda^{q+1}} \cap \mathcal{S}_{d}$ to 
$\Gamma^{2s+k,s,\lambda,\frac{\lambda}{2}}_{\delta,T,\varLambda^q}$ and from
$\Gamma^{2s+k,s,\lambda,\frac{\lambda}{2}}_{\delta+1,T,\varLambda^{q-1}}  \cap \mathcal{S}_{d}$ to 
$\Gamma^{2s+k,s,\lambda,\frac{\lambda}{2}}_{\delta,T,\varLambda^q}$ respectively, are also continuous. To do this, consider the norm
\begin{equation*}
\left\|\varPhi f\right\|_{\Gamma^{2s+k,s,\lambda,\frac{\lambda}{2}}_{\delta,T,\varLambda^q} }= 
\left\|\varPhi f\right\|_{C^{2s+k,s,\lambda,\frac{\lambda}{2}}_{\delta,T,\varLambda^q} }+
\left\|d \varPhi f\right\|_{C^{2s+k,s,\lambda,\frac{\lambda}{2}}_{\delta+1,T,\varLambda^{q+1}} }+
\left\|d^* \varPhi f\right\|_{C^{2s+k,s,\lambda,\frac{\lambda}{2}}_{\delta+1,T,\varLambda^{q-1}} }
\end{equation*}

Since $d \varPhi f = f$, the second term is $\left\| f\right\|_{C^{2s+k,s,\lambda,\frac{\lambda}{2}}_{\delta+1,T,\varLambda^{q+1}} }$.
And since $df=0$, then $d^* \varPhi f = 0$ by Lemma \ref{l.phi.Hoelder.d}.
Let us estimate the first term using embedding theorems, as well as the continuity of the operator $\varPhi$ from
$C^{2s+k-1,s,\lambda,0}_{\delta+1,T,\varLambda^{q+1}}$ to
$C^{2s+k,s,\lambda,0}_{\delta,T,\varLambda^q}$ and from 
$C^{2s+k-1,\lambda}_{\delta+1,\varLambda^{q+1}}$ to
$C^{2s+k,\lambda}_{\delta,\varLambda^q}$. Indeed,  
\begin{multline}
\label{ocendlfi}
\left\|\varPhi f\right\|_{C^{2s+k,s,\lambda,\frac{\lambda}{2}}_{\delta,T,\varLambda^q} } \leqslant 
\left\| f\right\|_{C^{2s+k,s,\lambda,0}_{\delta+1,T,\varLambda^{q+1}} }+ \\
+\sum\limits_{j=1}^{s}
\sup\limits_{t,\tau \in [0,T]}
\frac{\left\| \partial^j \left( f(\cdot,t) -  f(\cdot,\tau) \right) \right\|_{C^{2s+k-1,\lambda}_{\delta+1,\varLambda^{q+1}}}}
{\left| t-\tau \right|^{\frac{\lambda}{2}}}    
\leqslant
\left\| f\right\|_{C^{2s+k,s,\lambda,0}_{\delta+1,T,\varLambda^{q+1}} }+ \\
+\sum\limits_{j=1}^{s}
\sup\limits_{t,\tau \in [0,T]}
\frac{\left\| \partial^j \left( f(\cdot,t) -  f(\cdot,\tau) \right) \right\|_
{C^{2s+k,0}_{\delta+1,\varLambda^{q+1}}}}
{\left| t-\tau \right|^{\frac{\lambda}{2}}}  = 
\left\| f\right\|_{C^{2s+k,s,\lambda,\frac{\lambda}{2}}_{\delta+1,T,\varLambda^{q+1}} }
\end{multline}

The corresponding estimate showing the boundedness of the operator $\hat{\varPhi}$, is carried out in a similar way.

Thus, the first part of the theorem immediately follows from the continuity 
of the operators $\varPhi$, $\hat{\varPhi}$ and $(d,d^*)$, as well as the completeness of the 
spaces 
$\Gamma^{2s+k,s,\lambda,\frac{\lambda}{2}}_{\delta,T,\varLambda^q}$ and 
$\Gamma^{2s+k,s,\lambda,\frac{\lambda}{2}}_{\delta+1,T,\varLambda^{q+1} \oplus \varLambda^{q-1
}}$. 

We divide the rest of the proof into two statements.
\begin{lemma} 
\label{prozamk}
Let $s \in \mathbb{Z}_{\geqslant 0}$, $k \geqslant 1$, $0 < \lambda < 1$ and 
$n-1+m <\delta < n+m$. Then the set of pairs of forms $(f,g)$,
lying in the space $\Gamma^{2s+k,s,\lambda,\frac{\lambda}{2}}_{\delta+1,T,\varLambda^{q+1} \oplus \varLambda^{q-1}}$,
for which conditions \eqref{ortH} and \eqref{ortH2},
are valid is closed.
\end{lemma}

\begin{proof}
Consider a sequence $\left\{f_j,g_j\right\}_{j=1}^{\infty}$,
converging to $(f,g)$
in the space $\Gamma^{2s+k,s,\lambda,\frac{\lambda}{2}}_{\delta+1,T,\varLambda^{q+1} \oplus \varLambda^{q-1}}$,
such that conditions (\ref{ortH}) and (\ref{ortH2}) are satisfied for all $(f_j,g_j)$.

Since $\left\{f_j,g_j\right\}_{j=1}^{\infty}$ converges to a pair $(f,g)$ in the space 
$\Gamma^{2s+k,s,\lambda,\frac{\lambda}{2}}_{\delta+1,T,\varLambda^{q+1} \oplus \varLambda^{q-1}}$, then
$\left\{df_j,d^*g_j\right\}_{j=1}^{\infty}$ converges to $(df,d^*g)$ in 
$C^{2s+k,s,\lambda,\frac{\lambda}{2}}_{\delta+2,T,\varLambda^{q+2} \oplus \varLambda^{q-2}}$. But for any
$j \in \mathbb{N}$
the element $(df_j,d^*g_j)$
is equal to $0$.
This means that sequence $\left\{df_j,d^*g_j\right\}_{j=1}^{\infty}$
converges to $0$
in the space $C^{2s+k,s,\lambda,\frac{\lambda}{2}}_{\delta+2,T,\varLambda^{q+2} \oplus \varLambda^{q-2}}$.
Since the limit is unique, we can conclude that condition 
(\ref{ortH}) is satisfied.

Let us check the fulfillment of condition (\ref{ortH2}) for the pair $(f,g)$:
\begin{multline}
\label{K-B}
\big| (f,dh)_{L_2}+(g,d^*h)_{L_2} \big|\leqslant \\
\leqslant  \big| (f-f_j,dh)_{L_2} \big| + \big| (f_j,dh)_{L_2}\big| 
+\big| (g-g_j,d^*h)_{L_2} \big| + \big| (g_j,d^*h)_{L_2}\big| = \\
=\big| (f-f_j,dh)_{L_2} \big| +\big| (g-g_j,d^*h)_{L_2} \big|.
\end{multline}

Next, using \ref{K-B}) and  the Cauchy inequality we obtain
\begin{multline}
\label{vL2}
\big| (f-f_j,dh)_{L_2} \big| +\big| (g-g_j,d^*h)_{L_2} \big| \leqslant  \\
\leqslant
\left\| f-f_j \right\|_{L_2} \left\|dh\right\|_{L_2} +
\left\| g-g_j \right\|_{L_2} \left\|d^*h\right\|_{L_2} =  \\
=C_1 \left\| f-f_j \right\|_{L_2} + C_2 \left\| g-g_j \right\|_{L_2},
\end{multline}
and since the norm of $L_2$ is weaker than the norm of 
 $C^{2s+k,s,\lambda,\frac{\lambda}{2}}_{\delta+1,T,\varLambda^{q+1}}$, 
we see that the right hand side of (\ref{vL2}) does not exceed
\begin{equation*}
C_1 \left\| f-f_j \right\|_{C^{2s+k,s,\lambda,\frac{\lambda}{2}}_{\delta+1,T,\varLambda^{q+1}}} +
C_2 \left\| g-g_j \right\|_{C^{2s+k,s,\lambda,\frac{\lambda}{2}}_{\delta+1,T,\varLambda^{q-1}}}.
\end{equation*}

Thus, we obtain inequality
\begin{equation}
\label{eqz}
\big| (f,dh)_{L_2}+(g,d^*h)_{L_2} \big|\leqslant 
C_1 \left\| f-f_j \right\|_{C^{2s+k,s,\lambda,\frac{\lambda}{2}}_{\delta+1,T,\varLambda^{q+1}}} +
C_2 \left\| g-g_j \right\|_{C^{2s+k,s,\lambda,\frac{\lambda}{2}}_{\delta+1,T,\varLambda^{q-1}}}.
\end{equation}

Letting $j$ in (\ref{eqz})
to infinity, we get
\begin{equation*}
 (f,dh)_{L_2}+(g,d^*h)_{L_2} = 0.
\end{equation*}

This means that the subspace of $\Gamma^{2s+k,s,\lambda,\frac{\lambda}{2}}_{\delta+1,T,\varLambda^{q+1} \oplus \varLambda^{q-1}}$, consisting of the forms 
satisfying conditions (\ref{ortH}) and (\ref{ortH2}),
is closed. 
\end{proof}

It remains to prove that all elements of the image of operator (\ref{gtog})
satisfy conditions (\ref{ortH}) and (\ref{ortH2}).

\begin{lemma}
\label{proobr}
Let $s \in \mathbb{Z}_{\geqslant 0}$, $k \geqslant 1$, $0 < \lambda < 1$ and
$n-1+m <\delta < n+m$. Then the image of the operator \eqref{gtog}
is the set of forms $(f,g) \in \Gamma^{2s+k,s,\lambda, \frac{\lambda}{2}}_{\delta+1,T,\varLambda^{q+1}\oplus \varLambda^{q-1}}$,
for which conditions \eqref{ortH} and \eqref{ortH2} are valid.
\end{lemma}

\begin{proof}
Let $(f,g) \in \Gamma^{2s+k,s,\lambda, \frac{\lambda}{2}}_{\delta+1,T,\varLambda^{q+1}\oplus \varLambda^{q-1}} \subset
C^{2s+k,s,\lambda,0}_{\delta,T,\varLambda^{q+1} \oplus \varLambda^{q-1}}$,
then, according to Corollary
\ref{c.weight.Hoelder.d.easy}, the differential form
\begin{equation}
\label{ufif}
u= \varPhi f + \hat{\varPhi} g,
\end{equation}
belongs to the space $C^{2s+k+1,s,\lambda,0}_{\delta,T,\varLambda^q}$, and conditions (\ref{ortH}) and (\ref{ortH2}) are fulfilled for it.

Let us show that the form $u$
is contained in the space $\Gamma^{2s+k,s,\lambda, \frac{\lambda}{2}}_{\delta,T,\varLambda^q}$.
Indeed,
\begin{equation*}
\left\|u\right\|_
{\Gamma^{2s+k,s,\lambda, \frac{\lambda}{2}}_{\delta,T,\varLambda^q}}=
\left\|u\right\|_
{C^{2s+k,s,\lambda, \frac{\lambda}{2}}_{\delta,T,\varLambda^q}}+
\left\|du\right\|_
{C^{2s+k,s,\lambda, \frac{\lambda}{2}}_{\delta+1,T,\varLambda^{q+1}}}+
\left\|d^*u\right\|_
{C^{2s+k,s,\lambda, \frac{\lambda}{2}}_{\delta+1,T,\varLambda^{q-1}}}.
\end{equation*}

Since $df=0$ and $d^*g=0$, it is known from Lemma \ref{l.phi.Hoelder.d} that $d\hat{\varPhi} g = 0$ and
$d^*\varPhi f = 0$. Thus, the second and third terms in the last formula have the form
\begin{multline*}
\big\|du \big\|_
{C^{2s+k,s,\lambda, \frac{\lambda}{2}}_{\delta+1,T,\varLambda^{q+1}}}+
\big\|d^*u \big\|_
{C^{2s+k,s,\lambda, \frac{\lambda}{2}}_{\delta+1,T,\varLambda^{q-1}}}=\\
=\big\|d(\varPhi f + \hat{\varPhi} g)\big\|_
{C^{2s+k,s,\lambda, \frac{\lambda}{2}}_{\delta+1,T,\varLambda^{q+1}}}+
\big\|d^*(\varPhi f + \hat{\varPhi} g)\big\|_
{C^{2s+k,s,\lambda, \frac{\lambda}{2}}_{\delta+1,T,\varLambda^{q-1}}}=\\
=\big\|f  \big\|_
{C^{2s+k,s,\lambda, \frac{\lambda}{2}}_{\delta+1,T,\varLambda^{q+1}}}+
\big\| g  \big\|_
{C^{2s+k,s,\lambda, \frac{\lambda}{2}}_{\delta+1,T,\varLambda^{q-1}}}.
\end{multline*}

Using the representation 
$u=\varPhi f+ \hat{\varPhi} g$, in the norm  $\left\|u\right\|_
{C^{2s+k,s,\lambda, \frac{\lambda}{2}}_{\delta,T,\varLambda^q}}$   and carrying out an estimate similar to (\ref{ocendlfi}) , we obtain that
\begin{equation*}
\left\|u\right\|_
{C^{2s+k,s,\lambda, \frac{\lambda}{2}}_{\delta,T,\varLambda^q}} \leqslant
\left\|f\right\|_
{C^{2s+k,s,\lambda, \frac{\lambda}{2}}_{\delta+1,T,\varLambda^{q+1}}} +
\left\|g\right\|_
{C^{2s+k,s,\lambda, \frac{\lambda}{2}}_{\delta+1,T,\varLambda^{q-1}}}.
\end{equation*}

This means that $u$
is contained in the space $\Gamma^{2s+k,s,\lambda, \frac{\lambda}{2}}_{\delta,T,\varLambda^q}$.
And from (\ref{ufif})
we see that $u$ belongs to
the preimage of the forms $(f,g)$,
for which conditions (\ref{ortH}) and (\ref{ortH2}) are valid.
\end{proof}

Thus, Lemmas \ref{prozamk} and \ref{proobr} provide a proof of 
the second part of the theorem. 
\end{proof}

As in the previous sections, we define cocycles, coboundaries, and cohomology on spaces 
$\Gamma^{2s+k,s,\lambda,\frac{\lambda}{2}}_{\delta, \varLambda^{q}}$.
Let the set of coboundaries $Z^{2s+k,s,\lambda,\frac{\lambda}{2}}_{\delta, \varLambda^{q}}$
consist of all forms $f \in \Gamma^{2s+k,s,\lambda,\frac{\lambda}{2}}_{\delta, \varLambda^{q}}$, for which $df=0$,
and the set of cocycles 
$B^{2s+k,s,\lambda,\frac{\lambda}{2}}_{\delta, \varLambda^{q}}$ contains 
$f \in \Gamma^{2s+k,s,\lambda,\frac{\lambda}{2}}_{\delta, \varLambda^{q}}$ such that there exists 
$u \in \Gamma^{2s+k,s,\lambda,\frac{\lambda}{2}}_{\delta-1, \varLambda^{q-1}}$ satisfying the equation $du=f$. As usual, cohomology groups are defined as
$H^{2s+k,s,\lambda,\frac{\lambda}{2}}_{\delta, \varLambda^{q}}=Z^{2s+k,s,\lambda,\frac{\lambda}{2}}_{\delta, \varLambda^{q}} / B^{2s+k,s,\lambda,\frac{\lambda}{2}}_{\delta, \varLambda^{q}}$. Now we can consider 
the following complex of Banach spaces:
\begin{multline}
\label{complex3}
0 \rightarrow \Gamma^{2s+k,s,\lambda,\frac{\lambda}{2}}_{\delta, \varLambda^{0}} \xrightarrow{d_0}
\Gamma^{2s+k,s,\lambda,\frac{\lambda}{2}}_{\delta+1, \varLambda^{1}} \xrightarrow{d_1}\ldots \\
\ldots\xrightarrow{d_{n-2}} 
\Gamma^{2s+k,s,\lambda,\frac{\lambda}{2}}_{\delta+(n-1), \varLambda^{n-1}} \xrightarrow{d_{n-1}}
\Gamma^{2s+k,s,\lambda,\frac{\lambda}{2}}_{\delta+n, \varLambda^{n}} \xrightarrow{d_{n}} 0
\end{multline}

\begin{corollary}
\label{c.weight.Hoelder.d.not.easy.cor}
Let $n \geqslant 2$, $s,k\in {\mathbb Z}_+$, $0 < \lambda < 1$ and 
$f \in Z^{2s+k,s,\lambda,\frac{\lambda}{2}}_{\delta+1, \varLambda^{q+1}}$. If the form
$u \in \Gamma^{2s+k,s,\lambda,\frac{\lambda}{2}}_{\delta, \varLambda^{q}}$
is a solution to the equation $du=f$, 
then there exists a form $v \in \Gamma^{2s+k,s,\lambda,\frac{\lambda}{2}}_{\delta, \varLambda^{q}}$,
satisfying the system of operator equations
\begin{equation*}
\begin{cases}
dv=f, \\
d^*v=0.
\end{cases}
\end{equation*}
\end{corollary}

\begin{proof}
Since the space $C^{2s+k,s,\lambda,\frac{\lambda}{2}}_{\delta, \varLambda^{q}}$
is embedded continuously  into the space $C^{2s+k,s,\lambda,0}_{\delta, \varLambda^{q}}$,
the form $u$
can be represented in the form 
\begin{equation}
\label{viruchf}
u = d \varPhi  u + d^* \hat{\varPhi} u.
\end{equation}

Consider the norm of an operator $d^* \hat{\varPhi}$, acting on $\Gamma^{2s+k,s,\lambda,\frac{\lambda}{2}}_{\delta, \varLambda^{q}}$.
\small
\begin{equation*}
\big\| d^* \hat{\varPhi}u \big\|_{\Gamma^{2s+k,s,\lambda,\frac{\lambda}{2}}_{\delta, \varLambda^{q}} }=
\big\| d^* \hat{\varPhi} u \big\|_{C^{2s+k,s,\lambda,\frac{\lambda}{2}}_{\delta, \varLambda^{q}} }+
\big\| d d^* \hat{\varPhi} u \big\|_{C^{2s+k,s,\lambda,\frac{\lambda}{2}}_{\delta+1, \varLambda^{q+1}} } +
\big\| d^* d^* \hat{\varPhi} u \big\|_{C^{2s+k,s,\lambda,\frac{\lambda}{2}}_{\delta+1, \varLambda^{q-1}} }.
\end{equation*}
\normalsize
Here the last term is equal to zero, because $d^*\circ d^*=0$.
Using the embedding theorems and the continuity of the operator $ d^* \hat{\varPhi} $ on the spaces
$C^{2s+k,s,\lambda,0}_{\delta,T,\varLambda^q}$  and
$C^{2s+k,\lambda}_{\delta,\varLambda^q}$, we obtain an estimate for the first term: 
\begin{multline*}
\big\| d^* \hat{\varPhi} f \big\|_{C^{2s+k,s,\lambda,\frac{\lambda}{2}}_{\delta+1, \varLambda^{q-1}} }
\leqslant 
\big\| f \big\|_{C^{2s+k,s,\lambda,0}_{\delta+1, \varLambda^{q+1}} }+ \\
+\sum\limits_{j=1}^{s}
\sup\limits_{t,\tau \in [0,T]}
\frac{\big\| \partial^j \left( f (\cdot,t) -  f(\cdot,\tau) \right) \big\|_{C^{2s+k,0}_{\delta+1,\varLambda^{q+1}}}}
{\left| t-\tau \right|^{\frac{\lambda}{2}}} =
\big\| f \big\|_{C^{2s+k,s,\lambda,\frac{\lambda}{2}}_{\delta+1, \varLambda^{q+1}}. }
\end{multline*} 

To estimate the second term, we apply the operator $d$ to equation (\ref{viruchf}) and obtain
\begin{equation*}
\big\| d d^* \hat{\varPhi}u \big\|_{\Gamma^{2s+k,s,\lambda,\frac{\lambda}{2}}_{\delta+1, \varLambda^{q+1}} }=
\big\| f\big\|_{\Gamma^{2s+k,s,\lambda,\frac{\lambda}{2}}_{\delta+1, \varLambda^{q+1}} }.
\end{equation*}

Thus, $v=d^* \hat{\varPhi} u $ belongs to the space 
$\Gamma^{2s+k,s,\lambda,\frac{\lambda}{2}}_{\delta, \varLambda^{q}}$ and, similarly to Corollaries \ref{c.weight.Hoelder.d.easy.cor} and 
\ref{t.weight.Hoelder.d.cor}, is the desired form. 
\end{proof}
We can now describe the cohomology groups of complex (\ref{complex3}).
\begin{theorem}
Let $n \geqslant 2$, $s\in {\mathbb Z}_+$,  $0 < \lambda < 1$ and $n+m-1 < \delta < n+m$. Then the cohomology groups of complex \eqref{complex3} are isomorphic to the image of the operator $d\left(\varPhi-\varPhi_m\right)$acting from
$Z^{2s+k,s,\lambda,\frac{\lambda}{2}}_{\delta+1, \varLambda^{q+1}}$
to $Z^{2s+k,s,\lambda,\frac{\lambda}{2}}_{\delta+1, \varLambda^{q+1}}$.
\end{theorem}

\begin{proof}
The proof is carried out similarly to the proof of Theorem \ref{kohom1}, relying on Corollary 
\ref{c.weight.Hoelder.d.not.easy.cor} instead of Theorem \ref{t.weight.Hoelder.d}.
\end{proof}

Also, as in the previous case, the elements of the image of the operator $d\left( \varPhi-\varPhi_m \right)$
acting from $Z^{2s+k,s,\lambda,0}_{\delta+1, \varLambda^{q+1}}$ 
to $Z^{2s+k,s,\lambda,0}_{\delta+1, \varLambda^{q+1}}$ are represented as follows
\begin{equation*}
d \left( \sum_{\left|I\right|=q} \sum_{k=1}^{m+1} \sum_{j=1}^{J (k)} 
\frac{a_I (t) h_k^{(j)} \left(x\right) dx_I}
{(n+2k-2){\vartheta}^{n+2k-2}(x)} \right), 
\end{equation*}
where $a_I(t)$ are functions of the variable $t$ of the class $C^{s,\frac{\lambda}{2}}([0,T])$.

\textit{Acknowledgments.\,}
The author was supported by a grant of the Foundation for the advancement of theoretical
physics and mathematics ``BASIS.''



\begin{thebibliography}{99}

\bibitem{AT6}
Atiyah M.,
\textit{Elliptic operators, discrete groups and von Neumann
algebras}, 
Analyse et topologie. Asterisque. Vol.~32/33 (1976), P.~43--72.

\bibitem{Be11}
Behrndt T.,
\textit{On the Cauchy problem for the heat equation on Riemannian manifolds with conical 
singularities},
The Quarterly Journal of Math. V.~64 (2011), N.~4, P.~981--1007. 

\bibitem{BerMaj02}
Bertozzi A., Majda A.,
\textit{Vorticity and Incompressible Flows},
Cambridge University Press, Cambridge, 2002.

\bibitem{deRh55} 
De Rham G., 
\textit{Vari\'et\'es Diff\'erentiables}. 
Hermann$\&$C, \'Editeurs, 
Paris, 1955. 

\bibitem{Fri64}
Friedman A.,
\textit{Partial Differential Equations of Parabolic Type},
Prentice-Hall, Inc., Englewood Cliffs, NJ, 1964.

\bibitem{SidShl20} 
Gagelgans K.\,V., Shlapunov A.\,A.,
\textit{On the de Rham complex on a scale of anisotropic weghted H\"older spaces},
Siberian Electronic Mathematical Reports, V.~17 (2020), P.~428--444.


\bibitem{GKShw82}
Gol'dshtein V.\,M., Kuz'minov V.\,I., Shvedov I.\,A.,
\textit{Integration of differential forms of the classes 
$W^*_{p,q}$}, 
Siberian Mathematical Journal. V.~23 (1982), N.~5, P.~640--653.

\bibitem{Hoer65} 
H\"ormander L., 
\textit{$L^2$-estimates and existence theorems for the 
$\overline \partial$-operator}, Acta Math. 113, P.~89--152, 1965.

\bibitem{Kond66}
Kondrat'ev, V.\,A.,
\textit{Boundary problems for parabolic equations in closed domains}
Trans. Moscow Math. Soc. V.~15 (1966), P.~400--451.

\bibitem{Kry96} 
Krylov, N.\,V.,
\textit{Lectures on elliptic and parabolic equations in H\"older spaces. Graduate Studies in 
Math. V. 12}, AMS, Providence, Rhode Island, 1996.

\bibitem{LadSoUr67}
Ladyzhenskaya, O.\,A., Solonnikov V.\,A., Ural'tseva, N.\,N.,
\textit{Linear and Quasilinear Equations of Parabolic Type},
Moscow, Nauka,  1967.



\bibitem{MazRoss04}
Mazya V., Rossmann J.,
\textit{Schauder estimates for solutions to boundary value problems for second order elliptic systems in polyhedral domains},
Applicable Analysis. V.~83 (2004), N.~1, P.~271--308.

\bibitem{McOw79}
McOwen R.,
\textit{Behavior of the Laplacian on weighted Sobolev spaces},
Comm. Pure Appl. Math. V.~32 (1979), P.~783--795.

\bibitem{NireWalk73}
Nirenberg L., Walker H.,
\textit{The null spaces of elliptic partial differential operators in $\mathbb{R}^n$}
J. Math. Anal. and Appl. V.~42 (1973), P.~271--301.

\bibitem{Shla92}
Shlapunov A.\,A.,
\textit{The Cauchy problem for Laplace's equation},
Siberian Mathematical Journal. V.~33 (1992), N.~3, P.~534--542.

\bibitem{ShlTa18} 
Shlapunov~A., Tarkhanov N.,
\textit{An Open Mapping Theorem for the Navier-Stokes Equations},
Advances and Applications in Fluid Mechanics. V.~21 (2018), N.~2, P.~127--246.

\bibitem{SidShl19} 
Sidorova (Gagelgans) K.\,V., Shlapunov A.\,A.,
\textit{On the Closure of Smooth Compactly Supported Functions in Weighted H\"older Spaces},
Mathematical Notes. V.~105 (2019), N.~4. P.~604--617.

\bibitem{Tark95a}
Tarkhanov, N.,
\textit{Complexes of differential operators},
Dordrecht, Kluwer Ac. Publ., 1995.


\bibitem{Vod2010}
Vodop'yanov S.\,K., 
\textit{Spaces of differential forms and maps with controlled distortion},
Izvestiya: Mathematics. V.~74 (2010), N.~4, P.~663--689.



\end{thebibliography}
\end{document}